\documentclass{amsart}
\usepackage{amsmath}
\usepackage{amsthm}
\usepackage{amsfonts}
\usepackage{amssymb}
\usepackage{graphicx}
\usepackage{enumitem}
\usepackage{url}
\usepackage{color}
\usepackage{verbatim}
\theoremstyle{plain}
\newtheorem{theorem}{Theorem}[section]
\newtheorem{definition}[theorem]{Definition}
\newtheorem{lemma}[theorem]{Lemma}
\newtheorem{corollary}[theorem]{Corollary}
\newtheorem{proposition}[theorem]{Proposition}

\newtheorem{example}{Example}[section]
\newtheorem*{remark}{Remark}

\newcommand{\cy}[1]{\mathbb{Z}_{#1}}
\newcommand{\F}{\mathbb F}
\newcommand{\C}{\mathbb C}
\newcommand{\Z}{\mathbb Z}

\newcommand{\PG}{\mathrm{PG}}

\newcommand{\RN}[1]{%
  \textup{\uppercase\expandafter{\romannumeral#1}}%
}

 \def\zhou#1 {\fbox {\footnote {\ }}\ \footnotetext { From Yue: {\color{red}#1}}}

 \def\alex#1 {\fbox {\footnote {\ }}\ \footnotetext { From Alex: {\color{blue}#1}}}

\begin{document}
	\title{Cayley graphs of diameter two \linebreak from difference sets}
	\author[Alexander Pott]{Alexander Pott}
	\address{Faculty of Mathematics, Otto-von-Guericke University, Universitaetsplatz 2, D-39106 Magdeburg, Germany}
	\email{alexander.pott@ovgu.de}
	\author[Yue Zhou]{Yue Zhou}
	\address{Dipartimento di Mathematica e Applicazioni ``R. Caccioppoli", Universit\`{a} degli Studi di Napoli ``Federico \RN{2}", I-80126 Napoli, Italy}
	\email{yue.zhou.ovgu@gmail.com}
%	\date{\today}
	
%	\subjclass[2010]{51E21, 05B25, 12K10, 05E15}
	\maketitle
	
	\begin{abstract}
		Let $C(d,k)$ and $AC(d,k)$ be the largest order of a Cayley graph and a Cayley graph based on an abelian group, respectively, of degree $d$ and diameter $k$. When $k=2$, it is well-known that $C(d,2)\le d^2+1$ with equality if and only if the graph is a Moore graph. In the abelian case, we have $AC(d,2)\le \frac{d^2}{2}+d+1$. The best currently lower bound on $AC(d,2)$ is $\frac{3}{8}d^2-1.45 d^{1.525}$ for all sufficiently large $d$. In this paper, we consider the construction of large graphs of diameter $2$ using generalized difference sets. We show that $AC(d,2)\ge \frac{25}{64}d^2-2.1 d^{1.525}$ for sufficiently large $d$ and $AC(d,2) \ge \frac{4}{9}d^2$ if $d=3q$, $q=2^m$ and $m$ is odd. 
	\end{abstract}
		
\section{Introduction}\label{se:intro}
In a graph $\Gamma$, the \emph{distance} $d(u,v)$ from vertex $u$ to vertex $v$ is the length of a shortest $u$-$v$ path in $\Gamma$. The largest distance between two vertices in $\Gamma$ is the \emph{diameter} of $\Gamma$.  Let $\Gamma=(V,E)$ be a graph of maximum degree $d$ and diameter $k$. According to the Moore bound, $\Gamma$ has at most $M_{d,k}$ vertices, where
\[M_{d,k}=\left\{
  \begin{array}{ll}
    1+d \frac{(d-1)^k-1}{d-2}, & \hbox{if $d>2$;} \\
    2k+1, & \hbox{if $d=2$.}
  \end{array}
\right.
\]
When the order of $V$ equals $M_{d,k}$, the graph $\Gamma$ is called a \emph{Moore graph}. Clearly complete graphs ($k=1$) and cycles of odd order ($d=2$) are Moore graphs. 

The study of Moore graphs began with the work of Hoffman and Singleton \cite{hoffman_moore_1960}. It is not difficult to see that a Moore graph of diameter $k$ is always regular and its girth, namely the length of the shortest cycle contained in it, is $2k+1$. Furthermore, it can be shown that a Moore graph is distance regular. The Hoffman–-Singleton theorem states that any Moore graph with diameter $2$ must have valency $2$, $3$, $7$ or perhaps also $57$. The graphs corresponding to the first three valencies are the cycle of order $5$, the Petersen graph and the Hoffman--Singleton graph. The existence of a Moore graph with valency $57$ is still open. As proved by Damerell \cite{damerell_moore_1973} as well as Bannai and Ito \cite{bannai_1973_moore} independently, there are no other Moore graphs; see \cite[Section 23]{biggs_algebraic_1993} too.

As there are very few Moore graphs, it is interesting to ask the following so-called \emph{Degree/Diameter} problem. 
\begin{itemize}
	\item Given positive integers $d$ and $k$, find the largest possible number $N(d,k)$ of vertices in a graph with maximum degree $d$ and diameter $k$.
\end{itemize}
Since this is still quite a difficult problem, the following two problems have been investigated.
\begin{itemize}
	\item Find good upper bounds for $N(d,k)$ by proving nonexistence of graphs.
	\item Construct large graphs to increase the lower bounds for $N(d,k)$.
\end{itemize}
We refer to \cite{miller_moore_2013} for a recent survey on the Degree/Diameter problem.

By far, the best lower bounds for $N(d,2)$ follow from a construction  by Brown \cite{brown_graphs_1966}. The vertices of his graph are the set of points of $\PG(2,q)$, where $q$ is a prime power. Two different points $(a,b,c)$ and $(x,y,z)$ are adjacent if and only if $ax+by+cz=0$. This graph has $q^2+q+1$ vertices, it is not regular and its maximum degree $d=q+1$. Therefore $N(d,2)\ge d^2-d+1$ for $d=q+1$. By extending the Brown's graphs appropriately, we can get rid of the strong restriction on $d$ and show that $N(d,2)\ge d^2-2d^{1.525}$ for sufficiently large integer $d$; see \cite{siran_large_2011}. Clearly, this bound asymptotically approaches the Moore one.

Let $G$ be a group and $S\subseteq G$ such that $S^{-1}=S$ and $e\notin S$. Here $S^{-1}:= \{s^{-1}: s\in S\}$. The \emph{Cayley graph} $\Gamma(G,S)$ has a vertex set $G$, and two distinct vertices $g$, $h$ are adjacent if and only if $g^{-1}h\in S$. Here $S$ is called the \emph{generating set}. A Cayley graph is always vertex-transitive and regular, and its valency equals $\#S$. The following proposition gives us a strategy to construct Cayley graphs of diameter $k$.
\begin{proposition}\label{prop:cayley}
	The diameter of a Cayley graph $\Gamma(G,S)$ is $k$ if and only if $k$ is the smallest integer such that all elements in $G$ appear in  $\{\prod_{i=1}^k s_i: s_i\in S \text{ for } i=1,2,\dots,k\}$.
\end{proposition}

Cayley graphs have been extensively used in searching for lower bounds for $N(d,k)$. Actually several largest known graphs are Cayley graphs; see \cite{online_table_graphs}. By Proposition \ref{prop:cayley}, to construct Cayley graphs of diameter $2$, we need to find a subset $S\subsetneq G$ such that $G=\{s_1s_2: s_1,s_2\in S\}$. Let us use $C(d,k)$ to denote the largest order of Cayley graphs of valency $d$ and diameter $k$. By Proposition \ref{prop:cayley}, it is not difficult to see that $C(d,2)\le d(d-1)+d+1=d^2+1$, which coincides with the Moore bound.
Up to now, for $k=2$ the best result is obtained by \v{S}iagiov\'{a} and \v{S}ir\'{a}\v{n} in \cite{siagiova_approaching_2012}, in which it is proved that in a family of nonabelian groups there exist Cayley graphs of degree $d$, diameter $2$ and order larger than $d^2-6\sqrt{2}d^{3/2}$ for infinitely many $d$.  Hence this lower bound for Cayley graphs asymptotically approaches the Moore bound $d^2+1$ for all the graphs of diameter 2.
%$d\in \{2^{2m+\sigma}+(2+\sigma)2^{m+1}-6: m \in \Z^+, \sigma=0,1\}$. By adding involutions of $G$ to the generating set...

We use $AC(d,k)$ and $CC(d,k)$ to denote the largest order of Carley graphs of valency $d$ and diameter $k$ based on abelian groups and cyclic groups respectively. By simple counting argument, we see that $AC(d,2)\le d^2/2+d+1$. More general results on $AC(d,k)$ are obtained by Dougherty and Faber \cite{dougherty_degree-diameter_2004} thorough investigation of lattice coverings. For instance, they show that $AC(2\delta, k) \le \sum_{i=0}^{\delta}2^i\binom{\delta}{i}\binom{k}{i}$.
%stating that for even degrees $d=2\delta$, there exists a constant $c$ independent from $d$ and $k$ such that for any fixed $\delta\ge 2$ and all $k$,
%\[\frac{c\cdot 2^\delta}{\delta ! \delta (\ln \delta )^{1+\log_2 e}} k^\delta+O(k^{\delta-1})  \le AC(2\delta, k) \le \sum_{i=0}^{\delta}2^i\binom{\delta}{i}\binom{k}{i}. \]
Their approach works well in finding lower bounds for $AC(d,k)$ for small even values of $d$. It is proved that $AC(4,k)=2k^2+2k+1$ which reaches the above upper bound, and there is also a nice construction for $d=6$. On the other hand, for large $d$ and small $k$, the best general result  up to now is obtained by Macbeth, \v{S}iagiov\'{a} and \v{S}ir\'{a}\v{n} in \cite{macbeth_cayley_2012}:
\begin{equation}\label{eq:old_bound_AC}
	AC(d,2)>\frac{3}{8} d^2-4,
\end{equation}
where $d=4q-2$ for an odd prime power $q$. As Baker, Harman and Pintz proved in \cite{baker_difference_2001} that there is always a prime $p$ such that $p\in [x-x^{0.525},x ]$ for sufficiently large $x$, we may extend \eqref{eq:old_bound_AC} to all sufficiently large integers $d$, by simply adding more elements into the corresponding generating set; see \cite{siran_large_2011}. Based on a similar approach, in \cite{vetrik_abelian_2013} Vetr\'{i}k showed that
%\begin{equation}\label{eq:vetrik_bound_CC}
%	CC(d,2)\ge \frac{13}{36}(d+2)(d-4)
%\end{equation}
$CC(d,2)\ge \frac{13}{36}(d+2)(d-4)$ for any $d=6p-2$ where $p$ is a prime such that $p\neq 13$ and $p\not\equiv1 \pmod{13}$. 
%\v{S}ir\'{a}\v{n} \v{S}iagiov\'{a} and \v{Z}d\'{i}malov\'{a} 

Relative difference sets and direct product difference sets from finite Desarguesian planes play important roles in the construction of large Cayley graphs of diameter $2$ in \cite{macbeth_cayley_2012,vetrik_abelian_2013,siagiova_approaching_2012}. In Section \ref{sec:pre}, we give a short introduction to those generalized difference sets derived from finite projective planes. In Section \ref{sec:main}, we summarize the known approaches on constructing large Cayley graphs with diameter 2 based on abelian groups and present two results to improve the lower bound for $AC(d,2)$ for infinitely many $d$.

\section{Preliminaries}\label{sec:pre}
Let $G$ a group of order $v$ with the identity element $e$, and let $D$ be a $k$-subset of $G$. Then $D$ is called a \emph{$(v,k,\lambda)$-difference set} if the list of differences $d_1d_2^{-1}$ with $d_1,d_2\in D$, $d_1\neq d_2$, covers all elements in $G\setminus\{e\}$ exactly $\lambda$ times. There are various generalizations of difference sets, such as partial difference sets, relative difference sets, etc.; see \cite[Chapter 6]{beth_design_1999} and \cite{pott_finite_1995}. In this paper, we need the following general concept.

\begin{definition}\label{de:GneralizedDS}
    Let $G$ be a group of order $v$ and $N_1$, $\cdots$, $N_r$ subgroups of order $n_1,\dots,n_r$. Assume that $N_1$, $\cdots$, $N_r$ intersect pairwise trivially. A $(v; n_1,\dots, n_r;\allowbreak k, \lambda; \lambda_1,\dots,\lambda_r)$-\emph{generalized difference set}\index{difference set!generalized difference set (GDS)} (abbreviated to GDS) relative to the subgroups $N_i$ is a $k$-subset $D$ of $G$ such that the list of differences $d_1 d_2^{-1}$ with $d_1,d_2\in D$, $d_1\neq d_2$, covers all the elements in $G\backslash (N_1\cup N_2\cup \cdots \cup N_r)$ exactly $\lambda$ times, and the nonzero elements in $N_i$ exactly $\lambda_i$ times. The subgroups $N_i$ are called the \emph{exceptional} subgroups. %When $\lambda_i=0$, subgroups $N_i$ are also called the \emph{forbidden} subgroups. 
    A generalized difference set $D$ is called \emph{cyclic} or \emph{abelian} if $G$ has the respective property.

    Furthermore, if $r=1$, $\lambda_1=0$ and $v=mn$ where $n:=n_1$, then we call $D$ a \emph{relative difference set} with parameters $(m,n,k,\lambda)$ (an $(m,n,k,\lambda)$-RDS for short), and we call $N_1$ the \emph{forbidden subgroup}. If $N_1$ is a direct factor of $G$, the RDS is called \emph{splitting}.
\end{definition}

\begin{example}
    Let $\cy{n}$ denote the cyclic group of order $n$.
    \begin{enumerate}
        \item The set $\{\,1,2,4\,\}\subseteq \cy{7}$ is a $(7,3,1)$-difference set.
        \item The set $\{\,0,1\,\}\subseteq \cy{4}$ is a $(2,2,2,1)$-RDS relative to the forbidden subgroup $\{0,2\}$.
        \item The set $\{\,(1,2),(2,0),(0,3)\,\}$ is a $(16;4,4,4;3,1;0,0,0)$-GDS in $\cy{4}\times \cy{4}$ relative to the three exceptional subgroups $\cy{4}\times\{\,0\,\}$, $\{\,0\,\}\times \cy{4}$ and $\{\,(x,x):x\in\cy{4}\,\}$. \qed
    \end{enumerate}
\end{example}

Let $\C[G]$ denote the set of formal sums $\sum_{g\in G} a_g g$, where $a_g\in \C$ and $G$ is any (not necessarily abelian) group which we write here multiplicatively. We use ``$1$" to denote the identity element of $G$. The set $\C[G]$ is basically just a complex vector space whose basis is the set of group elements. We add these vectors componentwise, i.e.
$$\sum_{g\in G} a_g g +\sum_{g\in G} b_g g :=\sum_{g\in G} (a_g+b_g) g,$$
and we define a multiplication
$$(\sum_{g\in G} a_g g )\cdot (\sum_{g\in G} b_g g) :=\sum_{g\in G} (\sum_{h\in G} a_hb_{h^{-1}g})\cdot g.$$
Moreover,
$$\lambda \cdot(\sum_{g\in G} a_g g ):= \sum_{g\in G} (\lambda a_g) g $$
for $\lambda\in\C$.

For $D=\sum_{g\in G} a_g g$ and $E=\sum_{g\in G} b_g g$ in $\C[G]$, if $a_g$ and $b_g$ are all integers and $a_g\le b_g$ for each $g\in G$, then we write $D\preceq E$.

If $D=\sum_{g\in G} a_g g$, we define 
$$D^{(t)}:=\sum_{g\in G} a_g g^t.$$
An important case is $D^{(-1)}=\sum_{g\in G} a_g g^{-1}$. If $D$ is a subset of $G$, we identify $D$ with the group ring element $\sum_{g\in D} d$. The following result is straightforward.
\begin{lemma}\label{le:GDS_groupring}
    The set $D$ is a $(v; n_1,\dots,n_r;k,\lambda;\lambda_1,\dots,\lambda_r)$-GDS relative to the subgroups $N_i$ if and only if
    \begin{align}
        \label{eq:GDS_groupring}D\cdot D^{(-1)} =& k-(\lambda(1-r)+\lambda_1+\dots+\lambda_r)+\\
        \nonumber &\lambda(G-N_1-N_2-\dots-N_r)+\lambda_1N_1+\dots+\lambda_rN_r.\qed
    \end{align}
\end{lemma}

In \cite{dembowski_quasiregular_1967}, Dembowski and Piper have classified finite projective planes with large abelian collineation groups into eight cases. Several cases of them have close connections to generalized difference sets with $\lambda=1$. We refer to \cite{hughes_projective_1973} for an introduction of projective planes. Let $n$ and $G$ be the order of the corresponding projective plane and collineation group respectively. We summarize these (generalized) difference sets $D$ and the corresponding group ring equations \eqref{eq:GDS_groupring} in the following, which can be found in \cite{ganley_direct_1977,ganley_quasiregular_1975,ganley_relative_1975,ghinelli_finite_2003}.
\begin{enumerate}[label=(\Roman*)]
	\item\label{it:1} \textbf{Planar difference set}: Here $D$ is an $(n^2+n+1, n+1,1)$-difference set and equivalently
	\[D\cdot D^{(-1)} = n + G.\]
	\item\label{it:2} \textbf{Relative difference set}: Here $D$ is an $(n,n,n,1)$-RDS with a forbidden subgroup $N$ of order $n$ and equivalently
	\[D\cdot D^{(-1)} = n + G - N.\]
	Furthermore $D\cdot N = D^{(-1)}\cdot N =G$.
	\item\label{it:3} \textbf{Affine difference set}: Here $D$ is an $(n+1,n-1,n,1)$-RDS with a forbidden subgroup $N$ of order $n-1$ and equivalently
	\[D\cdot D^{(-1)} = n + G - N.\]
	Furthermore, $D\cdot N = D^{(-1)}\cdot N =G-M$,	where $M$ corresponds to a $(n-1)$-subset of $G$.
	\item\label{it:4} \textbf{Direct product difference set}:  Here $D$ is an $(n(n-1); n,n-1;n-1,1;0,0)$-GDS relative to subgroups $N_1$ and $N_2$ of orders $n$ and $n-1$, equivalently
	\[D\cdot D^{(-1)}= n + G -N_1 -N_2.\]
	Furthermore, $D \cdot N_1 = D^{(-1)} \cdot N_1= G$, $D \cdot N_2 = D^{(-1)} \cdot N_2= G - N_2$ and $N_1\cdot N_2=G$.
	\item\label{it:5} \textbf{Neofield}: Here $D$ is an $((n-1)^2;n-1,n-1,n-1;n-2,1;0,0,0)$-GDS relative to three subgroups $N_1$, $N_2$ and $N_3$, all of which are of order $n-1$ and intersect pairwise trivially. In group ring $\C[G]$, it can be equivalently written as
	\[D \cdot D^{(-1)} = n + G - N_1-N_2-N_3.\]
	Furthermore, $D\cdot N_i =D^{(-1)} \cdot N_i=G-M_i$ where $M_i$ are certain subsets of size $n-1$ in $G$ and $N_i\cdot N_j=G$ for $i,j=1,2,3$ and $i\neq j$.
\end{enumerate}

All known examples of these (generalized) difference sets come from projective planes, which are not necessarily desarguesian. Actually, there are generalized difference sets of types \ref{it:2}, \ref{it:3}, \ref{it:4} and \ref{it:5} contained in nonabelian groups; see \cite{jungnickel_automorphism_1982} for \ref{it:2} derived from non-commutative semifields, \cite{ganley_relative_1975} for \ref{it:3} from non-abelian collineation groups of the Desarguesian planes and \cite{hiramine_difference_1999} for \ref{it:4} and \ref{it:5} from nearfields.

\section{Main results}\label{sec:main}
In \cite{macbeth_cayley_2012,vetrik_abelian_2013}, the $(q,q,q,1)$-relative difference set
\[\{(x,x^2): x\in \F_q \}\subseteq (\F_q,+)\times (\F_q,+), \quad q \text{ odd}\]
and the direct product difference set
\[\{(x,x): x\in \F_q^*\} \subseteq (\F_q,+) \times (\F^*_q,*) \]
are used to construct large Cayley graphs of diameter $2$. A similar approach is applied in \cite{siagiova_approaching_2012} using a direct product difference set in a nonabelian group with many involutions.

The main idea of all these approach can be described as follows. Let $G$ be a group and $D\subseteq G$ one of the generalized difference sets listed above. There are $k$ subgroups $N_i$, $i=1,\dots,k$, where $k\le 3$. We know that the order of $G$ is approximately $n^2$ and the orders of $D$ and subgroups $N_i$ are approximately $n$. Let $H$ be an additively written abelian group and we consider $G\times H$. For $A\subseteq G$ and $h\in H$, let $(A,h)$ denote $\{(a,h): a \in A\}$. Let $\Psi$, $\Lambda_i$ ($i\le k$) be subsets of $H$. Viewed as an element in the group ring, we define $S\subseteq G\times H$ as
%\[ S := \{(D,g): g \in \Psi \} \cup \{(D^{(-1)},-g): g \in \Psi \} \cup \{(N_i, h_i): h_i \in\Lambda_i, 1\le i \le k\} \cup \Upsilon,\]
\[ S := \sum_{g \in \Psi}(D,g) + \sum_{g \in \Psi}(D^{(-1)},-g) + \sum_{i=1}^k\sum_{h\in \Lambda_i}(N_i, h)+ \Upsilon,\]
where $\Upsilon\subseteq G\times H$ is of small size compared with $n$. Furthermore, to construct undirected Cayley graphs, we want $S$ to be symmetric, i.e.\ viewed as an element in $\C[G]$, $S=S^{(-1)}$. Here it means that $\Lambda_i=\Lambda_i^{(-1)}$ and $\Upsilon=\Upsilon^{(-1)}$. There could be an overlapping between $(D,g)$ and $(D^{(-1)},-g)$ when $g=-g$. However, from the difference set property, it follows that there are at most 2 elements in the intersection of $D$ and $D^{(-1)}$.  Let $\psi = \#\Psi$ and $\theta_i=\#\Lambda_i$ for $1\le i\le k$. From the above analysis, we deduce that the size of $S$ is approximately $(2\psi+\sum_i\theta_i) n$.

Let us look at the elements in $S \cdot S$. There are several types of them:
\begin{itemize}
	\item $(D D^{(-1)},g_1-g_2)=(n+G-\sum N_i,g_1-g_2)$,
	\item $(D\cdot N_i, g+ h_i)=(G-M, g+h_i)$ and $(D^{(-1)}\cdot N_i, -g+ h_i)=(G-M, -g+h_i)$, where $M$ is of size $n-1$ or $0$ depending on $D$ and $i$. To be precise, for (\RN{1}), (\RN{2}) and $D\cdot N_1$ in (\RN{4}), $\#M=n-1$; for (\RN{3}), $D\cdot N_2$ in (\RN{4}) and (\RN{5}), $\#M=0$.
%	$(D\cdot N_k, j+i_k)=\left\{
%		  \begin{array}{ll}
%		    (G-M,j+i_k), & \hbox{$D$ is of type \ref{it:3} or \ref{it:4};} \\
%		    (G,j+i_k), & \hbox{otherwise.}
%		  \end{array}
%		\right.$\\
%		$(D^{(-1)}\cdot N_k, -j+i_k)=\left\{
%		  \begin{array}{ll}
%		    (G-M,-j+i_k), & \hbox{$D$ is of type \ref{it:3} or \ref{it:4};} \\
%		    (G,-j+i_k), & \hbox{otherwise.}
%		  \end{array}
%		\right.$
	\item $(N_i\cdot N_j, h_i+h_j)=(G,h_i+h_j)$ where $i\neq j$. 
	\item $(N_i \cdot N_i, h_i+h_i)= \#N_i (N_i, 2 h_i)$,
	\item $(D\cdot D, j_1+j_2)$ and $(D^{(-1)} \cdot D^{(-1)}, -j_1-j_2)$.
\end{itemize}
Let us look at the first component of these sets. For the first three ones, we see that almost every element in $G$ appears, which does not hold for the last two cases.

We want to use $S$ to define a Cayley graph of diameter $2$. By Proposition \ref{prop:cayley}, we have to show that every element in $G\times H$ can be written as $s_1 s_2$ where $s_1,s_2\in S$, i.e. \[G\times H \preceq S\cdot S \text{\quad in } \Z[G\times H].\]
Hence one strategy is to choose $\Psi$ and $\Lambda_i$ as small as possible such that 
\begin{equation}\label{eq:difference_cover_H}
	H \preceq \Psi \cdot \Psi^{(-1)} + \sum_{i=1}^k\Psi \Lambda_i + \sum_{i=1}^k\Psi^{(-1)} \Lambda_i + \sum_{i\neq j} \Lambda_i\Lambda_j\in \Z[H].
%	H = \{j_1-j_2 : j_1,j_2\in\Psi \}\cup \{j + i_k : j\in \Psi, i_k\in\Lambda_k, k\in \Omega \} \cup \{i_k + i_l : i_k\in \Lambda_k, i_l\in \Lambda_l, k\neq l \}.	
\end{equation}
Then we will see that most of the elements in $G\times H$ appear in the set of differences. For those exceptions, we may choose $\Upsilon$ carefully to generate more differences to cover them.
\begin{example}\cite[Theorem 2]{macbeth_cayley_2012}
	Let $D$ be a direct product difference set in $G=(\F_q,+) \times (\F^*_q,*)$ and $H:=\cy{6}$. Now the exceptional subgroups are $N_1=(\F_q,+) \times \{1\}$ and $N_2=\{0\} \times (\F^*_q,*)$. Let $\Psi := \{1\}$, $\Lambda_1=\{0\}$ and $\Lambda_2=\{3\}$. It is easy to check that \eqref{eq:difference_cover_H} holds. By choosing $\Upsilon=\{(0,1,1), (0,1,-1)\}$, it is routine to verify that $S\cdot S$ covers all the elements in $G\times H$.
\end{example}

Next we are going to present two constructions of Cayley graphs which improve the lower bound for $AC(d,2)$. 
\subsection{Construction \RN{1}}	
The first construction is based on neofields. Up to equivalence, the unique known $((n-1)^2;n-1,n-1,n-1;n-2,1;0,0,0)$-GDS in abelian groups exists in $(\F_q^*, *)\times(\F_q^*, *)$, where $n=q$. The three exceptional subgroups $N_1$, $N_2$ and $N_3$ are $\F_q^*\times \{1\}$, $\{1\}\times\F_q^*$ and $\{(x,x): x\in\F_q^* \}$ respectively. The generalized difference set is 
\[D:=\{(x,1-x): x \in \F_q, x\neq 0,1 \}.\]
It is straightforward to check that
\begin{eqnarray}
	\label{eq:neo:DN12} D \cdot N_i = D^{(-1)} \cdot N_i = G-N_i, \qquad\text{for }i =1,2& \text{and}\\
	\label{eq:neo:DN3} D\cdot N_3 = D^{(-1)} \cdot N_3 = G-\{(x,-x): x\in \F_q^*\}.&
\end{eqnarray}
Clearly $q$ is even if and only if $D\cdot N_3$ equals $G-N_3$.
\begin{theorem}\label{th:25/64}
	Let $q$ be a prime power and $d=\left\{
		  \begin{array}{ll}
		    8q-6, & \hbox{$q$ is even;} \\
		    8q-4, & \hbox{otherwise.}
		  \end{array}
		\right.$ Then
		\vspace*{2mm}
	\begin{equation}\label{eq:neo:25/64}
		AC(d,2)\ge \left\{ \def\arraystretch{2.2}
		  \begin{array}{ll}
		   \dfrac{25}{64}(d-2)^2, & \hbox{$q$ is even;} \\
		   \dfrac{25}{64}(d-4)^2, & \hbox{otherwise.}
		  \end{array}
		\right.
	\end{equation}
\end{theorem}
\begin{proof}
	Let $D:=\{(x,1-x): x \in \F_q, x\neq 0,1 \}$ and $\tilde{D}$ be defined by
	\[\tilde{D} := \left\{
	  \begin{array}{ll}
	    D\cup \{(1,1)\}, & \hbox{$q$ is even;} \\
	    D \cup \{(1,1),(1,-1)\}, & \hbox{otherwise.}
	  \end{array}
	\right.\]
	The exceptional subgroups are $N_1:=\F_q^*\times \{1\}$, $N_2:=\{1\}\times\F_q^*$ and $N_3:=\{(x,x): x\in\F_q^* \}$.
	By \eqref{eq:neo:DN12} and \eqref{eq:neo:DN3}, we see that in the group ring $\C[G]$,
	\begin{equation}\label{eq:neo:tildeD}
		\tilde{D} \cdot N_i = \tilde{D}^{(-1)} \cdot N_i \succeq G,
	\end{equation}
	for $i=1,2,3$.
	
	Now we define subset $S\subseteq \tilde{G}:=G\times \cy{5}\times \cy{5} $ as an element in the group ring $\C[\tilde{G}]$ by
	\begin{align*}
		S := &(\tilde{D},d)+(\tilde{D}^{(-1)},-d)+ (N_1,a_1)+(N_1,-a_1)+\\
		&(N_2,a_2)+(N_2,-a_2)+(N_3,a_3)+(N_2,-a_3),
	\end{align*}
	where $a_1=(1,0)$, $a_2=(0,1)$, $a_3=(0,2)$ and $d=(1,0)$. Hence, in the language used before, $\Psi=\{d\}$ and $\Lambda_i=\{a_i\}$ for $1\le i \le 3$.
	
	First, it is not difficult to see that
	\[(G,0,0)\le (\tilde{D},d)\cdot (\tilde{D}^{(-1)},d) +\sum_{i=1}^3(N_i,a_i)\cdot(N_i,-a_i).\]
	Second, one verifies
	\[ \{ \pm d\pm a_i: i=1,2,3 \}\cup \{ \pm a_i \pm a_j: i\neq j \} = \cy{5}\times \cy{5}\setminus \{(0,0)\}. \]
	Together with \eqref{eq:neo:tildeD}, we know that $S\cdot S$ covers all the elements in $\tilde{G}$.
	As
	\[ \#S = \left\{
	  \begin{array}{ll}
	    8q-6, & \hbox{if $q$ is even;} \\
	    8q-4, & \hbox{otherwise.}
	  \end{array}
	\right. \]
	which equals the valency $d$ of the graph, we have
	\[q-1=\left\{\def\arraystretch{2}
	  \begin{array}{ll}
	    \dfrac{d-2}{8}, & \hbox{if $q$ is even;} \\
	    \dfrac{d-4}{8}, & \hbox{otherwise.}
	  \end{array}
	\right.\]
	Together with the fact that the order of the graph equals $25(q-1)^2$, we get \eqref{eq:neo:25/64}.
\end{proof}
\begin{remark}
	It is natural to ask whether we can improve Theorem \ref{th:25/64} by choosing suitable subsets $\Psi$ and $\Lambda_i$ in $H$ which is not isomorphic to $\cy{5}\times\cy{5}$. We made an exhaustive computer search up to $\#H=50$, and there is no better result.
\end{remark}

\begin{corollary}
	For sufficiently large degree $d$,
	\[AC(d,2)\ge \frac{25}{64}d^2-2.1\cdot d^{1.525}.\]
\end{corollary}
\begin{proof}
	Let $S$ be the defining set of the Cayley graph in Theorem \ref{th:25/64} and let $q$ be odd. That means $\#S=8q-4$ and the graph has $25(q-1)^2$ vertices.
	
	For any integer $d\in [8q-4,25(q-1)^2-1]$, we can choose and add $(d-\#S)$ elements in $G$ into $S$ to get a new set $\tilde{S}$ such that $\#\tilde{S}=d$ and $\tilde{S}=\tilde{S}^{-1}$. Clearly the Cayley graph $\Gamma(G,\tilde{S})$ is still of diameter $2$.
	
	Now we fix $d$, which is sufficiently large. Let $b:=d/8+1/2$. By \cite{baker_difference_2001}, there is a prime $q$ such that $b-b^{0.525}\le q\le b$. Hence, we can take this $q$ and construct the Cayley graph $\Gamma(G,\tilde{S})$ such that $\#\tilde{S}=d$, and
	\begin{align*}
		\#G=25(q-1)^2 &= 25(b-b^{0.525}-1)^2\\
					  &> 25(b^2-2b^{1.525})\\
					  &> 25\left(\frac{d^2}{64}-2\left(\frac{d}{8}\right)^{1.525}\right)\\
					  &> \frac{25}{64}d^2-2.1\cdot d^{1.525}. \qedhere
	\end{align*}
\end{proof}
\subsection{Construction \RN{2}}	
The second construction is based on a special property of certain $(q,q,q,1)$-relative difference sets, where $q$ is even. This construction further improves the lower bound for $AC(d,2)$ in Theorem \ref{th:25/64}.
\begin{lemma}\label{le:ovalcover}
	Let $m$ be an odd positive integer and $q=2^m$. We represent $\cy{4}^m$ as $\F_{2^m}\times \F_{2^m}$ with the group operation
	\[(a,b)* (c,d)= (a+c,b+d+a\cdot c).\]
	Then the set $D=\{(x,0): x \in \F_{2^m}\}$ is a $(q,q,q,1)$-relative difference set in $\cy{4}^m$ with the forbidden subgroup $N=(\{0\}\times \F_{2^m}, *)$. Moreover, in the group ring $\C[\cy{4}^m]$, we have
	\[D\cdot D + D^{(-1)}\cdot D^{(-1)} = 2G.\]
\end{lemma}
\begin{proof}
	The proof that $D$ is a $(q,q,q,1)$-relative difference set can be found in \cite{pott_semifields_2014,schmidt_planar_2014,zhou_2^n2^n2^n1-relative_2012}. Actually $D$ corresponds to the trivial planar function $f(x)=0$ defined over $\F_{2^m}$ which gives rise to the Desarguesian plane of order $2^m$. For the readers' convenience, we repeat the proof.
	
	Noting that $(a,b)^{-1}=(a,b+a^2)$, we have
	\begin{align*}
		D\cdot D^{(-1)} & = \sum_{x\in\F_{2^m}} (x,0)\cdot \sum_{y\in\F_{2^m}} (y,y^2)\\	
						& = \sum_{x,y\in\F_{2^m}}(x+y, y^2+xy)\\
						& = \sum_{a\in\F_{2^m}}\sum_{y\in\F_{2^m}}(a, ay)\\
						& = q \cdot (0,0) + G-N,
	\end{align*}
	from which it follows that $D$ is a $(q,q,q,1)$-relative difference set and the forbidden subgroup is $N$. Furthermore,
	\begin{equation}\label{eq:DD}
		D\cdot D =  \sum_{x,y\in\F_{2^m}} (x+y,xy)= \sum_{a\in\F_{2^m}}(a, x^2+ax),
	\end{equation}
	and
	\begin{equation}\label{eq:D-1D-1}
		D^{(-1)}\cdot D^{(-1)} = \sum_{x,y\in\F_{2^m}} (x+y,x^2+xy+y^2) = \sum_{a\in\F_{2^m}}(a, x^2+ax+a^2).
	\end{equation}
	Let $L_a:\F_{2^m} \rightarrow \F_{2^m}$ be the mapping defined by $L_a(x)=x^2+ax$, where $a\in \F_{2^m}^*$. Clearly $L_a$ is additive and $L_a(x)=0$ has exactly two roots. It implies that the image sets of $L_a$ and of the mapping $x\mapsto L_a(x) + a^2$ are both of size $2^{m-1}$. Furthermore, these two image sets have no common element, because $m$ is odd and $x^2+ax+a^2=0$ has no root in $\F_{2^m}$. Hence
	\[\sum_{x\in\F_{2^m}}(a, x^2 + ax) + \sum_{x\in\F_{2^m}}(a, x^2 + ax+a^2)
	= 2 \sum_{x\in \F_{2^m}}(a,x).\]
	It implies that $D\cdot D + D^{(-1)}\cdot D^{(-1)} = 2G$.
\end{proof}

\begin{theorem}\label{th:new_AC_bound}
	Let $d=3q$, where $q=2^m$ with $odd$ m. Then
	\[AC(d,2) \ge \frac{4}{9}d^2.\]
\end{theorem}
\begin{proof}
	Let $D$ be the $(q,q,q,1)$-relative difference set defined in $\cy{4}^m$ as in Lemma \ref{le:ovalcover}. We denote the forbidden subgroup by $N$. The subset $S\subseteq G\times \cy{4}=\cy{4}^{m+1}$ is defined as an element in the group ring as
	\[S:=(D,1) + (D^{(-1)},-1) + (N,0) \in \C[\cy{4}^{m+1}]. \]
	Then
	\begin{align*}
		S\cdot S =& (D\cdot D^{(-1)}, 0) + q(N, 0) + (D\cdot D, 2) + (D^{(-1)}\cdot D^{(-1)},2)\\
				  & + (D \cdot N, 1) + (D^{(-1)} \cdot N, -1) \\
				 =& q(0,0)+(q-1)(N,0)+(G,\cy{4}\setminus\{2\})+(D\cdot D, 2) + (D^{(-1)}\cdot D^{(-1)},2)\\
				 =& q(0,0)+(q-1)(N,0)+(G,\cy{4}\setminus\{2\})+2(G, 2) \qquad (\text{Lemma \ref{le:ovalcover}})\\
				 =& q(0,0)+(q-1)(N,0)+(G, 2)+(G,\cy{4}).
	\end{align*}
	Hence $S\cdot S$ covers all the elements in $\cy{4}^{m+1}$. As the valency $d$ of $\Gamma(\cy{4}^{m+1},S)$ is $\#S=3q$, we have $q=d/3$ and the order of this graph is $4q^2=\frac{4}{9}d^2$.
\end{proof}

In the end, we consider the possibility to improve the lower bound for $AC(d,2)$ using the approach in Theorem \ref{th:new_AC_bound}. Let $H$ be an additively written abelian group and $\#H=l$. Let $G=\cy{4}^m$ and $D$ a $(2^m,2^m,2^m,1)$-relative difference set in $G$. Let $S\subseteq G\times H$ be defined as
\[ S := \sum_{g \in \Psi}(D,g)+\sum_{g \in \Psi}(D^{(-1)},-g) + \sum_{h\in \Lambda} (N, h).\]
Let $s:=2\psi+\theta=\# S$, where $\psi=\#\Psi$ and $\theta = \#\Lambda$. By counting the elements in $S\cdot S$, we see that the inequality
\[l\le 1+\psi (\psi-1) + \frac{\psi(\psi+1)}{2} +  2\psi \theta\]
must be satisfied.

Now let us consider the maximum value of 
\[\tau:=\frac{1+\psi (\psi-1) + \frac{\psi(\psi+1)}{2} +  2\psi \theta}{s^2}.\]
For given $s$, we have
\begin{align*}
	\tau & \le  \frac{1+\psi(\psi-1) +  \frac{\psi(\psi+1)}{2} +  2\psi(s-2\psi)}{s^2}\\
		 & =(-\frac{5}{2}\psi^2 + (2s-\frac{1}{2})\psi + 1) / s^2\\
		 & \le \left(\frac{1}{10}\left(2s-\frac{1}{2}\right)^2 + 1\right) /s^2,
\end{align*}
which is smaller than $4/9$ when $s\ge 4$. Therefore, it is impossible to improve the result in Theorem \ref{th:new_AC_bound} by using the same approach with any other groups $H$  of size larger than $4$.

\section*{acknowledgment}
The second author is supported by the Research Project of MIUR (Italian Office for University and Research) ``Strutture geometriche, Combinatoria e loro Applicazioni" 2012.

%\bibliographystyle{abbrv}
%\bibliography{D:/Documents/References/Reference_math}
\end{document}